\documentclass[10pt]{article}

 \usepackage[usenames,dvipsnames]{pstricks}
 \usepackage{epsfig}
 \usepackage{pst-grad} 
 \usepackage{pst-plot} 
 \usepackage[space]{grffile} 
 \usepackage{etoolbox} 
 \makeatletter 
 \patchcmd\Gread@eps{\@inputcheck#1 }{\@inputcheck"#1"\relax}{}{}
 \makeatother
\usepackage{mathrsfs}
\usepackage{dsfont}
\usepackage[usenames,dvipsnames]{pstricks}
\usepackage{epsfig}
\usepackage{pst-grad} 
\usepackage{pst-plot} 
\usepackage[space]{grffile} 
\usepackage{etoolbox} 
\makeatletter 
\patchcmd\Gread@eps{\@inputcheck#1 }{\@inputcheck"#1"\relax}{}{}
\makeatother
\usepackage[usenames,dvipsnames]{pstricks}
\usepackage{epsfig}
\usepackage{pst-grad} 
\usepackage{pst-plot} 
\usepackage[space]{grffile} 
\usepackage{etoolbox} 
\makeatletter 
\patchcmd\Gread@eps{\@inputcheck#1 }{\@inputcheck"#1"\relax}{}{}
\makeatother
\usepackage{amsmath}
\usepackage{amssymb}
\usepackage{amsthm}
\usepackage[latin1]{inputenc}
\usepackage{graphicx}
\usepackage[colorinlistoftodos]{todonotes}
\usepackage[usenames,dvipsnames]{pstricks}
\usepackage{epsfig}
\usepackage{pst-grad} 
\usepackage{pst-plot} 

\usepackage{comment}
\usepackage{ifthen}
\usepackage{color}

\usepackage{marginnote}

\newcommand{\intav}[1]{\mathchoice {\mathop{\vrule width 6pt height 3 pt depth  -2.5pt
\kern -8pt \intop}\nolimits_{\kern -6pt#1}} {\mathop{\vrule width
5pt height 3  pt depth -2.6pt \kern -6pt \intop}\nolimits_{#1}}
{\mathop{\vrule width 5pt height 3 pt depth -2.6pt \kern -6pt
\intop}\nolimits_{#1}} {\mathop{\vrule width 5pt height 3 pt depth
-2.6pt \kern -6pt \intop}\nolimits_{#1}}}

\def\polhk#1{\setbox0=\hbox{#1}{\ooalign{\hidewidth\lower1.5ex\hbox{`}\hidewidth\crcr\unhbox0}}}

\renewcommand{\div}{\operatorname{div}}

\newcommand{\osc}{\operatorname{osc}}

\renewcommand{\div}{\operatorname{div}}

\newtheorem{theorem}{Theorem}[section]

\newtheorem{lemma}{Lemma}[section]
\newtheorem{corollary}{Corollary}[section]
\newtheorem{proposition}{Proposition}[section]
\newtheorem{remark}{Remark}[section]
\newtheorem{assumption}{A}

\begin{document}

\title{Regularity of solutions to a class of variable--exponent fully nonlinear elliptic equations}
\author{Anne C. Bronzi,\;\; Edgard A. Pimentel, \smallskip \\ Giane C. Rampasso, \;\; and \;\; Eduardo V. Teixeira}
\date{}

\newcommand{\Addresses}{{
  \bigskip
  \footnotesize

  \noindent Anne C. Bronzi, \textsc{Department of Mathematics, University of Campinas,
    Campinas, Brazil.}\par\nopagebreak
  \textit{E-mail address}, A. Bronzi: \texttt{annebronzi@ime.unicamp.br}

  \medskip

\noindent Edgard A. Pimentel, \textsc{Department of Mathematics, Pontifical Catholic University of Rio de Janeiro, Rio de Janeiro, Brazil.}\par\nopagebreak
  \textit{E-mail address}, E. Pimentel: \texttt{pimentel@puc-rio.br}

  \medskip

\noindent Giane C. Rampasso, \textsc{Department of Mathematics, University of Campinas,
    Campinas, Brazil.}\par\nopagebreak
  \textit{E-mail address}, G. Rampasso: \texttt{girampasso@ime.unicamp.br}

  \medskip

\noindent Eduardo V. Teixeira, \textsc{Department of Mathematics, University of Central Florida,
    Orlando, Florida, USA}\par\nopagebreak
  \textit{E-mail address}, E. Teixeira: \texttt{eduardo.teixeira@ucf.edu}

  \medskip

}}

\maketitle

\begin{abstract}

\noindent In the present paper, we propose the investigation of variable-exponent, degenerate/singular elliptic equations in non-divergence form.  This current endeavor parallels the by now well established theory of functionals satisfying nonstandard growth condition, which in particular encompasses problems ruled by  the $p(x)$-laplacian operator. Under rather general conditions, we prove viscosity solutions to variable exponent fully nonlinear elliptic equations are locally of class $C^{1,\kappa}$ for a universal constant $0< \kappa < 1$. A key feature of our estimates is that they do not depend on the modulus of continuity of exponent coefficients, and thus may be employed to investigate a variety of problems whose ellipticity degenerates and/or blows-up in a discontinuous fashion.

\medskip


\medskip 



\end{abstract}

\vspace{.1in}
	
\section{Introduction}
	
We investigate regularity of viscosity solutions to variable-exponent, fully nonlinear elliptic equations of the form
\begin{equation}\label{mainequationIntro}
	F(x, Du, D^2u) \,=\,f(x) \;\;\;\;\; \mbox{in}\;\;\;\;\; B_1,
\end{equation}
where $F(x, p, M)$ is elliptic with respect to the matrix argument $M$; however ellipticity may degenerate and/or blow-up at different rates along the critical region $\mathscr{C}:=\{ p =0\}$ --- which, in heuristic terms, operates as if it were a sort of {\it abstract free boundary } of the problem.  Indeed, we are interested in partial differential equations (PDEs) of the form \eqref{mainequationIntro} for which $D_M F(x,p,M) \sim |p|^{\theta(x)}$. 

\smallskip 

A meaningful prototypical example we have in mind concerns equations of the form
\begin{equation}\label{mainequation}
	|Du|^{\theta(x)} F(D^2u) \,=\,f(x) \;\;\;\;\; \mbox{in}\;\;\;\;\; B_1,
\end{equation}
for a $(\lambda,\Lambda)$-elliptic operator $F$ and a variable exponent given by a function $\theta \colon B_1\to\mathbb{R}$ satisfying minimal conditions to be specified later. Our goal is to establish differentiability and local H\"older continuity of the gradient for solutions of such a class of equations. 

\smallskip 

We are particularly interested in obtaining such estimates with no continuity assumption on the variable exponent $\theta$. We also aim at a comprehensive regularity theory which allows $\theta$ to impel degenerate and singular characters on the diffusibility of the governing operator. Among applications we have in mind are two-phase free boundary problems prescribing different degeneracy laws in each phase. The governing operator for such a problem would be
$$
	\mathscr{F}(x,u, Du, D^2u) = \left | Du \right |^{p_{-} \chi_{\{u\le 0\}} + p_{+} \chi_{\{u > 0\}} } F(D^2u) = G(X,u),
$$
which can be treated by approximation from operators studies in this current work. For such an endeavor, it is critical to establish estimates that do not depend on the continuity of variable exponents. 

\smallskip 

Degenerate elliptic PDEs with variable exponents appear na\-turally in several branches of applied mathematics, as an attempt to adjust diffusibility of the process more efficiently.  An emblematic application comes from the theory of image enhancements.  The variational approach to image processing consists in examining a true image $u \colon B_1\subset\mathbb{R}^2\to\mathbb{R}$ as the minimizer of a functional with variable diffusibility attributes. 
A toy-model in this setting is given by the minimization problem
\[
	 \int \left ( \left|Du\right|^{\theta(x)}\,+\,\frac{\sigma}{2}\left|u\,-\,d\right|^2  \right ) dx \longrightarrow \text{min}, 
\]	
where $d \colon B_1\subset\mathbb{R}^2\to\mathbb{R}$ is the so called measured image. It is defined as $d(x)\,:=\,u(x)\,+\,\zeta$, with $\zeta$ denoting a noise term. 
Applications of this sort has fostered a well established variational theory to investigate functionals of the general form
\begin{equation}\label{eq_var1}
	I[u]\,:=\,\int_{B_1}L(x,Du(x))dx,
\end{equation}
under a natural growth condition on the Lagrangian  $L \colon B_1\times\mathbb{R}^d\to\mathbb{R}$, namely:
\begin{equation}\label{eq_var2}
	\left|p\right|^{\theta(x)}\,\leq\,L(x,p)\,\leq\,C\left(1\,+\,\left|p\right|^{\theta(x)}\right),
\end{equation}	
with $\theta(x)>1$. We refer the reader to \cite{levine06} and the references therein. Different choices of $\theta(x)$ account for distinct algorithms in image denoising and reconstruction. For example, the case $\theta(x)\equiv 1$ is called ROF; see \cite{rof}. This method is known for preserving edges. From a technical viewpoint, the choice $\theta(x)\equiv 1$ suggests a mathematical treatment in the space of functions of bounded variation $BV(B_1)$, which is a desirable feature of the algorithm, for edge reconstruction requires the minimizers to be discontinuous functions.  A drawback of this technique is referred in literature as staircasing; i.e., the presence of noise in smooth regions of the image may lead to piecewise constant regions in the processed image. To bypass this issue, an alternative is to prescribe $\theta(x)\equiv 2$. Although it prevents staircasing effect, this choice falls short in image reconstruction for it fails in preserving edges. It is clear that the choice of fixed $\theta(x)\equiv \theta$ satisfying $1<\theta<2$  should import features of both regimes. However, being fixed, the parameter would favor either the reconstruction of smooth regions or the edges preservation.

\smallskip 

In order to circumvent this rigidity, the natural alternative is to allow the exponent $\theta$ to depend on the space variable $x\in B_1$. The first toy-model incorporating this feature appears in \cite{blomgren}; the author proposes a functional of the form
\[
	\int_{B_1}\left|Du\right|^{\theta\left(\left|Du\right|\right)}dx,
\]
where $\theta:\mathbb{R}\to\mathbb{R}$ is such that $\theta(t)\to 2$ as $t\to0$ and $\theta(t)\to 1$ as $t\to\infty$. For completeness, we mention that variational exponents appear also in the context of electrorheological fluids \cite{mathphys1,mathphys2,mathphys3} and the thermistor problem \cite{zhikovther}; see also \cite{antontsev1994} and the references therein.

\smallskip 

Regularity issues pertaining to the minimization problem \eqref{eq_var1}--\eqref{eq_var2} are, from the perspective of numerical analysis,  paramount, whereas from the mathematics viewpoint, rather delicate. With respect to the latter, a number of important developments has been obtained in the literature. If $\omega$ denotes the  modulus of continuity of $p(x)$, lower and higher regularity estimates on local minimizers depend in a critical way on $\omega$. Indeed, under the assumption
\begin{equation}\label{eq_modcon0}
	\limsup_{r\to0}\,\omega(r)\ln\left(\frac{1}{r}\right) \le  L  < +\infty,
\end{equation}
it is possible to obtain higher integrability of the gradient of minimizers, i.e. $D u \in L^q$, as well as local H\"older continuity of $u$, for some $0<\alpha \ll 1$, see for instance \cite{AF, V96}.  Under slightly stronger assumption than \eqref{eq_modcon0}, namely taking $L=0$,  {\sc Arcebi and Mingione} proved in \cite{acerbi_mingione01} that minimizers to \eqref{eq_var1} are locally $\alpha$-H\"older continuous, for every $\alpha\in(0,1)$.

\smallskip 

The results mentioned above refer to lower level regularity, as in the constant exponent case, a classical result due to {\sc Ural'tseva} assures local minimizers are $C^{1,\alpha}$, for some $0<\alpha \ll 1$. The question about minimal assumptions on $\omega$ as to develop a $C^{1,\alpha}$ regularity theory for  local minimizers of  \eqref{eq_var1}--\eqref{eq_var2} was also settled in \cite{acerbi_mingione01}, where authors show local H\"older continuity of the gradient of minimizers provided $\theta(x)$ is H\"older continuous and $L(x,p)$ is twice-differentiable with respect to $p$ in $B_1\times\left(\mathbb{R}^d\setminus\left\{0\right\}\right)$; see also \cite{coscia_mingione} for earlier results of that sort.

\smallskip 

The case of systems is the subject of \cite{acerbi_mingione_pisa}. In that paper, the authors produce a partial regularity result for the minimizers of \eqref{eq_var1}--\eqref{eq_var2}. In fact, under H\"older continuity of the exponent and regularity conditions on the Lagrangian $L$, the authors prove the H\"older continuity of the gradient in a subset $\Omega\subset B_1$, satisfying $\left|B_1\setminus\Omega\right|=0$. Here, minimizers are taken in the Sobolev space $W^{1,1}_{loc}(B_1,\mathbb{R}^d)$.

\smallskip 

The regularity of minimizers to \eqref{eq_var1}--\eqref{eq_var2} has an obvious counterpart in the regularity of the (weak) solutions to the associated Euler-Lagrange equation. In addition, the $\theta(x)$-growth regime suggests the use of Lebesgue and Sobolev spaces with variable exponents; see \cite{diening1}.

\smallskip 

In line with this observation, the results in \cite{acerbi_mingione_crelle} intersect those two approaches. In that paper, the authors examine the regularity of the solutions to 
\begin{equation*}\label{eq_euler1}
	-\div a(x,Du)\,=\,-\div\left(\left|f(x)\right|^{\theta(x)-2}f(x)\right)\;\;\;\;\;\mbox{in}\;\;\;\;\;B_1,
\end{equation*}
where $a(x,p)$ and $f$ are given and satisfy a set of conditions concerning growth and regularity. A typical assumption on $a(x,p)$ would be
\begin{equation*}\label{eq_euler2}
	\left|a(x,p)\right|\,\leq\,C\left(1\,+\,\left|p\right|^2\right)^\frac{\theta(x)-1}{2}.
\end{equation*}
Under such a growth condition, and log-H\"older continuity of $\theta(x)$, the authors prove $\left|Du\right|^{\theta(x)}\in L^q_{loc}(B_1)$, provided $\left|f\right|^{\theta(\cdot)}\in L^q_{loc}(B_1)$, for $q>1$.

\smallskip 

More recently, \cite{xianlingfan} studied the regularity of the solutions to
\[
	\div a(x,Du)\,=\,f(x,Du)\;\;\;\;\;\mbox{in}\;\;\;\;\;B_1
\]
under growth and regularity conditions on $a(x,p)$ and the nonlinearity $f(x,p)$. In particular, $a$ and $f$ satisfy a $\theta(x)$-growth condition, where this exponent is supposed to be H\"older continuous. The author proves that solutions are of class ${C}^{1,\alpha}_{loc}(B_1)$, for some $\alpha\in(0,1)$ unknown.

\smallskip 

While the theorems described above are deep and sharp,  it is relevant to highlight that in the variational theory no regularity results are known (or even valid) when $\theta$ is assumed to be just bounded and measurable.

\smallskip 

Regularity theory for degenerate/singular fully nonlinear equations in non-divergence form:
\begin{equation}\label{eq1}
	\left|Du\right|^p F(D^2u)\,=\,f\;\;\;\;\;\mbox{in}\;\;\;\;\;B_1,
\end{equation}
has also attracted great attention in the past years, see \cite{BD1, birregrad, DFQ, Imb-Silv1, damedugle, ART2, TeixPota} among several other works on this subject. By now we have a fairly good understanding of the underlying regularity theory for solutions of equation \eqref{eq1}. Local H\"older continuity was proven independently in \cite{DFQ} and \cite{I} by means of Alexandroff-Bakelman-Pucci estimate; see also \cite{IS2} for a much more robust treatment.  Upon boundedness assumption on the source function $f$, {\sc Imbert and Silvestre} proved in \cite{Imb-Silv1} that solutions are locally of class $C^{1,\alpha}$, for some $0< \alpha \ll 1$. Under convexity assumption on the map $M \mapsto F(M)$, the optimal gradient H\"older exponent for all solutions to \eqref{eq1} turns out to be precisely  $\alpha := \frac{1}{1\,+ p}$; see \cite{damedugle} for details.

\smallskip 

In this present article, we launch the study of nonvariational PDEs of the form \eqref{mainequation}; however presenting variable-exponents such as in the variational theory accounted in  \eqref{eq_var1}--\eqref{eq_var2}. 
Indeed, under minimal conditions we will prove that viscosity solutions to 
\[
	\left|Du\right|^{\theta(x)} F(D^2u)\,= f(x),\;\;\;\;\;\mbox{in}\;\;\;\;\;B_1,
\]
are of class $C^{1,\alpha}$ in $B_{1/2}$. Localized and sharp estimates will also be produced as a combination of  analytic and geometric tools. Hence, this present article should be understood as both a generalization of the ${C}^{1,\alpha}$ regularity estimate known for the constant-exponent equation, e.g. \cite{birreg,Imb-Silv1,damedugle}, as well as a parallel endeavor to the variable-exponent variational theory, e.g. \cite{coscia_mingione,acerbi_mingione01,xianlingfan}.

\smallskip 

As we conclude this introduction, let us briefly comment on techniques used in the proofs. While the strategy put forward in this article has certainly been greatly benefitted by previous works, such as \cite{birreg,Imb-Silv1, IS2, damedugle}, several new difficulties had to be overcome by means of new ideas and tools. As a way of example, to establish results that do not depend on the continuity of $\theta(x)$, required that geometric regularity transmission methods to be embedded in a much finer setting, see Section  \ref{sct TP}. Also, that our estimates allow the PDE to alternate between the degenerate and the singular regimes, in different regions of the domain, entails additional layers of complexity and required a detailed treatment.  In turn the comprehensive investigation put  forward in this work does involve much more robust analysis as to unveil the contribution of each regime to the geometry of the solutions. We believe the ideas and methods introduced here are bound to be applicable in a number of other related problems.

\smallskip 

Consequential to our tangential approach is the pointwise improved regularity of the solutions. Under continuity assumptions on $\theta(x)$, we examine the growth of the solutions \emph{at a particular point} of the domain. Here, continuity assumptions build upon the pointwise behavior of the exponent to rule out regime-switching, as the ellipticity vanishes. In this scenario, the tangential analysis recovers enhanced, optimal information on the regularity of the solutions, in the form of a pointwise description of the growth rate at the particular point.

\smallskip 

The remainder of this paper is organized as follows: Sections \ref{sec_assump} and \ref{sec_preliminary}  present some basics of the theory and details our main assumptions. In Section \ref{sec_compactness} we put forward lower regularity results; those aim at producing compactness for the solutions to auxiliary problems appearing further in our arguments. Section \ref{sct TP} yields a geometric tangential path connecting the regularity theory of variable coefficient equation to the uniform elliptic one. Section \ref{sec_c1alpha} details the proof of Theorem \ref{teo_c1alpha}. Finally, improved pointwise regularity is the subject of Section \ref{Sharp sct}.

\bigskip

\noindent {\bf Acknowledgments:} AB is funded by CNPq-Brazil (\#312119/2016-0). EPl is partially funded by CNPq-Brazil (\#307500/2017-9 and 433623/2018-7) by FAPERJ-Brazil (\# E-26/201.609/2019) and by baseline/start-up funds from the Department of Mathematics at PUC-Rio. GR is funded by CNPq-Brazil (\#140674/2017-9). ET is partially sponsored by senior faculty UCF start-up grant. This study was also financed in part by the Coordenação de Aperfeiçoamento de Pessoal de Nível Superior - Brasil (CAPES) - Finance Code 001.

\section{Assumptions and main results}\label{sec_assump}

In this section, we describe the main assumptions of the paper.  First, we detail the conditions imposed on the fully nonlinear operator $F$:
\begin{assumption}[Ellipticity of $F$]\label{assump_operatorF}
We suppose $F:\mathcal{S}(d)\to\mathbb{R}$ is uniformly $(\lambda,\Lambda)$-elliptic, i.e.,
\[
	\lambda\left\|N\right\|\,\leq\,F(M+N)\,-\,F(M)\,\leq\,\Lambda\left\|N\right\|,
\]
for some $0<\lambda\leq \Lambda$, and every $M,\,N\in\mathcal{S}(d)$ with $N\geq 0$. In addition, we assume, with no loss of generality, $F(0) = 0$.
\end{assumption}

For simplicity, we shall restrict the analysis to the theory of continuous viscosity solutions. Thus, source function $f(x)$ as well as variable exponent $\theta(x)$ will always be assumed continuous. However, an important aspect of our results is that they do not depend upon continuity of the variable exponent $\theta$. This is critical for applications we have in mind. We require, nonetheless, a lower bound on $\theta$, as described below:

\begin{assumption}[Lower bound of the variable exponent]\label{assump_theta1}
We suppose $\theta\in{C}(\overline{B_1})$ satisfies
\[
	\inf_{x\in B_1}\theta(x)\,>\,-1.
\]
\end{assumption}

We comment once more that differentiability estimates proven in this article will depend solely upon $L^\infty$-norm of $\theta$, and thus are independent of any modulus of continuity on the exponent function. To highlight this fact, we introduce the following notation:

$$
	\theta^{+}(x) := \max \{0, \theta(x)\} \quad \text{and} \quad \theta^{-}(x) := \min \{0, \theta(x)\}.
$$

The relevance of $\theta^+$ and $\theta^-$ also lies in unveiling how the regularity profile varies as the PDE switches regime from degenerate to singular and vice-versa.

We recall a consequence of Krylov-Safonov Harnack inequality, see \cite{C1}, is that viscosity solutions to the homogeneous equation $F(D^2h) = 0$ are locally of class $C^{1,\bar\alpha}$ for a universal $0< \bar\alpha < 1$, i.e. depending only on dimension, and ellipticity constants $\lambda$ and $\Lambda$. Furthermore
$$
	\|h\|_{C^{1,\bar{\alpha}}(B_{1/2})} \le C \|h\|_{L^\infty(B_1)},
$$
for another universal constant $C>1$. 

We are ready to state our main theorem.

\begin{theorem}[H\"older regularity of the gradient]\label{teo_c1alpha}
Let $u\in {C}(B_1)$ be a viscosity solution to \eqref{mainequation}. Suppose A\ref{assump_operatorF} and A\ref{assump_theta1} are in force.
Then $u\in {C}^{1,\gamma}_{loc}(B_1)$, for all  $\gamma$ verifying 
$$
	\gamma < \min \left \{\bar{\alpha}, \frac{1}{1+\|\theta^{+}\|_\infty + \|\theta^{-}\|_\infty } \right \}.
$$
In addition, there exists a universal constant $C>0$, such that
\[
	 \|u\|_{L^\infty(B_{1/2})} + \sup\limits_{\substack{x,y \in B_{1/2} \\ x\not = y}}  \frac{\left | \nabla u(x) - \nabla u(y) \right | }{|x-y|^\gamma} \leq C\left(1+\left\|u\right\|_{L^\infty(B_1)}\,+\,\left\|f\right\|^{\frac{1}{1+\inf_{B_1}\theta}}_{L^\infty(B_1)}\right).  
\]
\end{theorem}

\begin{remark}
The optimal regularity in Theorem \ref{teo_c1alpha} accounts for a diffusion process that degenerates \emph{and} blows up in different subregions of the domain. In particular, it unveils the precise contribution of both regimes to the regularity of the solutions. As we should expect, a hybrid process yields lower regularity levels, when compared to  a purely degenerate or purely singular diffusion. We will return to this issue in the last Section of this work.
\end{remark}

\begin{remark}
We notice that Theorem \ref{teo_c1alpha} recovers the sharp regularity obtained in \cite{damedugle}, for the fixed-exponent case. Indeed, if $\theta(x)\equiv\theta$, then either $\theta^+\equiv 0$ or $\theta^-\equiv 0$. In any case, the restriction on the exponent $\gamma$ above reads
\[
	\gamma < \min \left \{\bar{\alpha}, \frac{1}{1+\theta}\right \}.
\]
\end{remark}
The next section puts forward some elementary notions and gathers a few auxiliary results.

\section{Preliminary material}\label{sec_preliminary}

In this section we collect some elementary notions and former results to be used in the paper. We start with the definition of the Pucci extremal operators. 

Let $0<\lambda\leq\Lambda$ be fixed. The Pucci extremal operators $\mathcal{M}^{\pm}_{\lambda,\Lambda}:\mathcal{S}(d)\to\mathbb{R}$ are given by
\[
	\mathcal{M}^+_{\lambda,\Lambda}(M)\,:=\,\Lambda\sum_{e_i>0}e_i\,+\,\lambda\sum_{e_i<0}e_i
\]
and
\[
	\mathcal{M}^-_{\lambda,\Lambda}(M)\,:=\,\lambda\sum_{e_i>0}e_i\,+\,\Lambda\sum_{e_i<0}e_i,
\]
where $(e_i)_{i=1}^d$ are the eigenvalues of $M$. Notice that $(\lambda,\Lambda)$-ellipticity can be stated in terms of $\mathcal{M}^{\pm}_{\lambda,\Lambda}$: an operator $F \colon \mathcal{S}(d)\to\mathbb{R}$ is $(\lambda,\Lambda)$-elliptic if
\begin{equation}\label{eq_ellipticpucci}
	\mathcal{M}^-_{\lambda,\Lambda}(N)\,\leq\,F(M\,+\,N)\,-\,F(M)\,\leq\,\mathcal{M}^+_{\lambda,\Lambda}(N),
\end{equation}
for every $M,\,N\in\mathcal{S}(d)$, with $N\geq0$.

The inequalities in \eqref{eq_ellipticpucci} build upon the so-called Ishii-Lions Lemma to produce initial levels of compactness for the solutions; see Section \ref{sec_compactness}. Next we recall the Ishii-Lions result. 

Let $G:B_1\times \mathbb{R}^d\times\mathcal{S}(d)\to\mathbb{R}$ be a $(\lambda,\Lambda)$-elliptic operator and $u$ a normalized viscosity solution for the equation 
\[
G(x,Du,D^2u)=0.
\]
The Ishii-Lions Lemma reads as follows:

\begin{proposition}[Ishii-Lions Lemma]\label{lemma_ishiilions} 
Let $\Omega\subset B_1$ and $\psi$ be twice continuously differentiable in a neighborhood of $\Omega\times \Omega$. Define $v:\Omega \times \Omega\to\mathbb{R}$ as
\[
	v(x,y)\,:=\,u(x)\,-\,u(y).
\]
Suppose $(\overline{x},\overline{y})\in \Omega\times \Omega$ is a local maximum of  $v\,-\psi$ in $\Omega\times \Omega$. Then, for each $\varepsilon>0$, there exist matrices $X$ and $Y$ in $\mathcal S(d)$ such that
\[
	G\left(\overline{x},D_x\psi(\overline{x},\overline{y}),X\right)\,\leq\,0\,\leq\,G\left(\overline{y},-D_y\psi(\overline{x},\overline{y}),Y\right),
\]
and the matrix inequality
\begin{equation*}
-\left(\frac{1}{\varepsilon}+\|A\|\right)I\,\leq\,
\left(
\begin{array}{ccc}
X   & 0 \\
0  &-Y 
\end{array}
\right)
\leq 
A+\varepsilon A^2
\end{equation*}
holds true, where $A=D^2\psi\left(\overline{x},\overline{y}\right)$.
\end{proposition}

For a proof of Proposition \ref{lemma_ishiilions}, we refer the reader to \cite[Theorem 3.2]{crandall1992user}. In our concrete case, the operator $G$ takes the form
\[
	G(x,p,M):=|p|^{\theta(x)}F(M)-f(x).
\]
In the next section we study the compactness of the solutions to 
\[
	\left|q\,+\,Du\right|^{\theta(x)}F(D^2u)\,=\,f\;\;\;\;\;\mbox{in}\;\;\;\;\;B_1.
\]
Clearly, the case $q\equiv 0$ accounts for \eqref{mainequation}. These results will be of the utmost relevance in a series of sequential arguments to appear further in the paper. 

To conclude this section, we examine the smallness regime imposed on the problem. Namely, we explicitly verify that it is possible to suppose
\begin{equation}\label{eq_smallnessregime}
	\left\|u\right\|_{L^\infty(B_1)}\,\leq\,1,\;\;\;\;\;\;\;\;\;\;\mbox{and}\;\;\;\;\;\;\;\;\;\;\left\|f\right\|_{L^\infty (B_1)}\,<\,\varepsilon,
\end{equation}
for some $0<\varepsilon\ll 1$ without loss of generality.

\subsection{Smallness regime}\label{subsec_small}

Here we explore the structure of \eqref{mainequation}. In particular, we examine its scaling properties that allow us to work under the conditions in \eqref{eq_smallnessregime}. {Suppose that $u$ solves the equation \eqref{mainequation} and consider $v:B_1\to\mathbb{R}$ defined as
\[
	v(x)\,:=\,\frac{u(\tau x\,+\,x_0)}{K},
\]
where $x_0\in B_1$ is fixed and $\tau$ and $K$ are positive constants to be determined. Easily one checks that $v$ solves
\[
	\left|Dv\right|^{\overline{\theta}(x)}\overline{F}(D^2v)\,=\,\overline{f}\;\;\;\;\;\mbox{in}\;\;\;\;\;B_1,
\]
where
\[
	\overline{\theta}(x)\,:=\,\theta(\tau x\,+\,x_0),
\]
\[
	\overline{F}(M)\,=\,\frac{\tau^2}{K}F\left(\frac{K}{\tau^2}M\right)
\]
and
\[
	\overline{f}(x)\,=\,\frac{\tau^{2+\overline{\theta}(x)}}{K^{1+\overline{\theta}(x)}}f(\tau x\,+\,x_0).
\]
Note $\overline{F}$ is still a $(\lambda,\Lambda)$-elliptic operator. In addition, as $\theta>-1$, it follows that $\overline{f}\in L^\infty(B_1)$.  By setting
\[
	K\,:=\,1+\left\|u\right\|_{L^\infty(B_1)}\,+\,\left\|f\right\|^{\frac{1}{1+\inf_{B_1}\theta}}_{L^\infty(B_1)}
\]
and
\[
	\tau\; := \varepsilon^{\frac{1}{2+\inf_{B_1} \theta}},
\]
we ensure that $v$ solves an equation in the same class as \eqref{mainequation}, and it is under the smallness regime prescribed in \eqref{eq_smallnessregime}. Estimates proven for $v$ gets transported to $u$ by factors that depend explicitly on $K$ and $\tau$.

\section{Compactness for a family of degenerate PDEs}\label{sec_compactness}

In this section we obtain local H\"older continuity estimates for viscosity solutions to \eqref{mainequation}. Such estimates yield compactness with respect to uniform convergence to  a  large class of functions related to the equation we propose to study. Those levels of compactness unlock the geometrical structure along which we transport regularity properties from the homogeneuous problem 
\[
   F(D^2u)=0 \ \ \ \mbox{in} \ \ \ B_1
\]
to the solutions of \eqref{mainequation}. The reasoning employed in this subsection is inspired by the methods put forward in \cite{Imb-Silv1}, though it requires extra care.


\begin{lemma}[H\"older continuity of the solutions]\label{compactness_large}
Let $u\in {C}(B_1)$ be a normalized viscosity solution to
\begin{equation*}
	|q+Du|^{\theta(x)}F(D^2u)=f(x) \,\,\,\, \mbox{in} \,\,\,\ B_1,
\end{equation*}
where $q\in\mathbb{R}^d$ is an arbitrary vector. Suppose A\ref{assump_operatorF} and A\ref{assump_theta1} hold. Then $u\in {C}^\beta_{loc}(B_1)$, for every $\beta\in(0,1)$. Moreover, there is $C>0$, depending only on dimension, ellipticity constants and $\beta$,  such that
\[
	\left\|u\right\|_{{C}^\beta(B_{1/2})}\leq C\left(1+\left\|u\right\|_{L^\infty(B_1)}\,+\,\left\|f\right\|^{1/(1\,+\,\inf\theta(x))}_{L^\infty(B_1)}\right).
\]
\end{lemma}

\begin{proof}
To prove the local H\"older regularity of $u$ let us fix $0<r<1$ and verify that there exists positive numbers $L_1$ and $L_2$ such that
\begin{equation}\label{eq_nespresso}
	\mathfrak{L}:=\displaystyle\sup_{x,y\in B_r}\left(u(x)-u(y)-L_1|x-y|^{\beta}-L_2(|x-x_0|^2+|y-x_0|^2)\right)\leq 0
\end{equation}
 for every $x_0\in B_{1/2}$. As it is usual, we argue by contradiction. That is, we suppose there exists $x_0\in B_{1/2}$ for which $\mathcal{L}>0$, for every $L_1>0$ and $L_2>0$.
 
We introduce two auxiliary functions $\phi,\,\psi:\overline{B_r}\times\overline{B_r}\to\mathbb{R}$, to be defined as
 \[
 	\psi(x,y)\,:=\, L_1\left|x\,-\,y\right|^\beta\,+\,L_2\left(\left|x\,-\,x_0\right|^2\,+\,\left|y\,-\,x_0\right|^2\right)
 \]
 and
 \[
 	\phi(x,y)\,:=\,u(x)\,-\,u(y)\,-\,\psi(x,y).
 \]
 We denote by $(\overline{x},\overline{y})$ a maximum point for $\phi$; i.e., 
 \[
 	\phi(\overline{x},\overline{y})\,=\,\mathcal{L}\,>\,0
 \]
 and
 \[
 	 L_1\left|\overline{x}\,-\,\overline{y}\right|^\beta\,+\,L_2\left(\left|\overline{x}\,-\,x_0\right|^2\,+\,\left|\overline{y}\,-\,x_0\right|^2\right)\,\leq\,\osc_{B_1}u\,\leq\,2
 \]
 Before we proceed, we set
 \[
 	L_2\,:=\,\left(\frac{4\sqrt{2}}{r}\right)^2.
 \]
 The former choice of $L_2$ implies  
 \[
 	\left|\overline{x}\,-\,x_0\right|\,+\,\left|\overline{y}\,-\,x_0\right|\,\leq\,\frac{r}{2}.
 \]
Hence, we conclude $\overline{x}$ and $\overline{y}$ are in $B_r$. In addition, $\overline{x}\neq \overline{y}$; otherwise, $\mathcal{L}\leq 0$ trivially.
 
 For ease of presentation, we split the proof in three steps. First, the Ishii-Lions Lemma builds upon \eqref{eq_nespresso} to produce a viscosity inequality. 
 
 \bigskip
 
\noindent{\bf Step 1} We start off by invoking the Ishii-Lions Lemma (see Proposition \ref{lemma_ishiilions}) as to assure the existence of a limiting sub-jet $(q_{\bar x},X)$ of $u$ at $\bar x$ and a limiting super-jet $(q_{\bar y},Y)$ 
of $u$ at $\bar y,$ where
\[
  q_{\bar x}:= D_x\psi(\bar x,\bar y)=L_1\beta|\bar x-\bar y|^{\beta-2}(\bar x-\bar y)+2L_2(\bar x-x_0)
\]
and
\[
  q_{\bar y}:= -D_y\psi(\bar x,\bar y)=L_1\beta|\bar x-\bar y|^{\beta-2}(\bar x-\bar y)-2L_2(\bar y-x_0)
\]
such that the matrices $X$ and $Y$ verify the inequality
\begin{equation}\label{matrix_inequality1}
  \left(
  \begin{array}{ccc}
  X   & 0 \\
  0  &-Y 
  \end{array}
  \right)
  \leq 
  \left(
  \begin{array}{ccc}
  Z  & -Z \\
  -Z  & Z
  \end{array}
  \right)+(2L_2+\iota)I,
\end{equation}
for 
\[
	Z:= L_1\beta|\bar x-\bar y|^{\beta-4}\left(|\bar x-\bar y|^2I-(2-\beta)(\bar x-\bar y)\otimes(\bar x-\bar y)\right)
\]
where $0<\iota$ only depends on the norm of $Z$ and can be made sufficiently small. 

For vectors of the form $(z,z)\in\mathbb{R}^{2d}$, we apply the matrix inequality \eqref{matrix_inequality1} to obtain
\[
  \langle(X-Y)z,z\rangle\leq(4L_2+2\iota)|z|^2.
\]
We conclude that all the eigenvalues of $(X-Y)$ are below $4L_2+2\iota.$ On the other hand, applying \eqref{matrix_inequality1} to the particular vector 
\[
  \bar{z}:=\left(\frac{\bar x-\bar y}{|\bar x-\bar y|},\frac{\bar y-\bar x}{|\bar x-\bar y|}\right),
\]
we get
\[
  \left\langle(X-Y)\frac{\bar x-\bar y}{|\bar x-\bar y|},\frac{\bar x-\bar y}{|\bar x-\bar y|}\right\rangle\leq(4L_2+2\iota-4L_1\beta(1-\beta)|\bar x-\bar y|^{\beta-2})\left|\frac{\bar x-\bar y}{|\bar x-\bar y|}\right|^2.
\]
Then, at least one eigenvalue of $(X-Y)$ is below $4L_2+2\iota-4L_1\beta(1-\beta)|\bar x-\bar y|^{\beta-2}$. This quantity will be negative for large values of $L_1.$ By the definition of the minimal Pucci operator we get
\[
  \mathcal{M}^{-}_{\lambda,\Lambda}(X-Y)\geq 4\lambda L_1\beta(1-\beta)|\bar x-\bar y|^{\beta-2}-(\lambda+(d-1)\Lambda)(4L_2+2\iota).
\]

From the two viscosity inequalities
\[
  |q+q_{\bar x}|^{\theta(\bar x)}F(X)\leq f(\bar x)
\]
and 
\[
  |q+q_{\bar y}|^{\theta(\bar y)}F(Y)\geq f(\bar y),
\]
and from the uniformly ellipticity
\[
  F(X)\geq F(Y)+\mathcal{M}^{-}_{\lambda,\Lambda}(X-Y),
\]
we obtain
\begin{equation}\label{eq_masterineq}
\begin{array}{lll}
  4\lambda L_1\beta(1-\beta)&\leq& (\lambda+(d-1)\Lambda)(4L_2+2\iota)|\bar x-\bar y|^{2-\beta}\\
  &+& f(\bar x)|\bar x-\bar y|^{2-\beta}|q+q_{\bar x}|^{-\theta(\bar x)}\\
  &-& f(\bar y)|\bar x-\bar y|^{2-\beta}|q+q_{\bar y}|^{-\theta(\bar y)}.
\end{array}
\end{equation}

In what follows we shall distinguish two cases. First we consider $\left|q\right|<A_0$, where $A_0>0$ is a constant to be determined later in the proof; in Step 3 we consider the complementar case.

\bigskip

\noindent{\bf Step 2}  Suppose $\left|q\right|\,\leq\,A_0$. Notice that, if $\theta(\bar x)>0$ then
\begin{equation*}
\begin{aligned}
  f(\bar x)|\bar x-\bar y|^{2-\beta}|q+q_{\bar x}|^{-\theta(\bar x)}&\leq \frac{\|f\|_{L^{\infty}(B_1)}|\bar x-\bar y|^{2-\beta}}{(L_1\beta|\bar x-\bar y|^{\beta-1}-2L_2|\bar x-x_0|-|q|)^{\theta(\bar x)}}\\
  &\leq \frac{\|f\|_{L^{\infty}(B_1)}|\bar x-\bar y|^{2-\beta+(1-\beta)\theta(\bar x)}}{(L_1\beta-2L_2|\bar x-x_0||\bar x-\bar y|^{1-\beta}-A_0|\bar x-\bar y|^{1-\beta})^{\theta(\bar x)}}\\
  &\leq \frac{\|f\|_{L^{\infty}(B_1)}}{(L_1\beta-2L_2-A_0)^{\theta(\bar x)}}
\end{aligned}
\end{equation*}
since $|\bar x-\bar y|$ and $|\bar x-x_0|$ are uniformly bounded by $1$, $(2-\beta+(1-\beta)\theta(\bar x))\geq 0$ and by choosing $L_1$ large enough the term $L_1\beta-2L_2-A_0$ is positive.

Similarly, if $\theta(\bar y)>0$ we obtain
\[
  -f(\bar y)|\bar x-\bar y|^{2-\beta}|q+q_{\bar y}|^{-\theta(\bar y)}\leq \frac{\|f\|_{L^{\infty}(B_1)}}{(L_1\beta-2L_2-A_0)^{\theta(\bar y)}}.
\]

Therefore, if both $\theta(\bar x)$ and $\theta(\bar y)$ are positive we have that
\begin{align}\label{contradiction2_case1}
  4\lambda L_1\beta(1-\beta)&\leq (\lambda+(d-1)\Lambda)(4L_2+2\iota)+\frac{\|f\|_{L^{\infty}(B_1)}}{(L_1\beta-2L_2-A_0)^{\theta(\bar x)}}\\\nonumber&\quad+\frac{\|f\|_{L^{\infty}(B_1)}}{(L_1\beta-2L_2-A_0)^{\theta(\bar y)}}.
\end{align}

On the other hand, if $-1<\theta(\bar x)\leq 0$, the Young's inequality yields
\begin{equation*}
\begin{aligned}
  &f(\bar x)|\bar x-\bar y|^{2-\beta}|q+q_{\bar x}|^{-\theta(\bar x)}\leq \\
  &\leq \|f\|_{L^{\infty}(B_1)}|\bar x-\bar y|^{2-\beta}(L_1\beta|\bar x-\bar y|^{\beta-1}+2L_2|\bar x-x_0|+|q|)^{-\theta(\bar x)}\\
  &\leq\|f\|_{L^{\infty}(B_1)}|\bar x-\bar y|^{2-\beta+(\beta-1)\theta(\bar x)}(L_1\beta+(2L_2|\bar x-x_0|+A_0)|\bar x-\bar y|^{1-\beta})^{-\theta(\bar x)}\\
  &\leq \|f\|_{L^{\infty}(B_1)}2^{-\theta(\bar x)}((L_1\beta)^{-\theta(\bar x)}+(2L_2+A_0)^{-\theta(\bar x)})\\
  &\leq \lambda\beta(1-\beta)L_1+C(\lambda,\beta)\|f\|_{L^{\infty}(B_1)}^{1/(1+\theta(\bar x))}+\|f\|_{L^{\infty}(B_1)}(4L_2+2A_0)^{-\theta(\bar x)},
\end{aligned}
\end{equation*}
since $|\bar x-\bar y|$ and $|\bar x-x_0|$ are uniformly bounded by $1$ and 
\[
	(2-\beta+(1-\beta)\theta(\bar x))\geq0.
\]
Similarly, if $-1<\theta(\bar y)\leq 0$
\begin{equation*}
\begin{aligned}
  &-f(\bar y)|\bar x-\bar y|^{2-\beta}|q+q_{\bar y}|^{-\theta(\bar y)}\\
  &\leq \lambda\beta(1-\beta)L_1+C(\lambda,\beta)\|f\|_{L^{\infty}(B_1)}^{1/(1+\theta(\bar y))}+\|f\|_{L^{\infty}(B_1)}(4L_2+2A_0)^{-\theta(\bar y)}.
\end{aligned}
\end{equation*}
Therefore, if both $\theta(\bar x)$ and $\theta(\bar y)$ belongs to $(-1,0]$, we find that
\begin{equation}\label{contradiction2_case2}
\begin{aligned}
  2\lambda L_1\beta(1-\beta)\leq& (\lambda+(d-1)\Lambda)(4L_2+2\iota)\\&+C(\lambda,\beta)(\|f\|_{L^{\infty}(B_1)}^{1/(1+\theta(\bar x))}+\|f\|_{L^{\infty}(B_1)}^{1/(1+\theta(\bar y))})\\
  &+\|f\|_{L^{\infty}(B_1)}((4L_2+2A_0)^{-\theta(\bar x)}+(4L_2+2A_0)^{-\theta(\bar y)}).
\end{aligned}
\end{equation}

Finally, if $\theta(\bar x)>0$ and $-1<\theta(\bar y)\leq 0$ we obtain
\begin{equation}\label{contradiction2_case3}
\begin{aligned}
  3\lambda L_1\beta(1-\beta)\leq& (\lambda+(d-1)\Lambda)(4L_2+2\iota)+\frac{\|f\|_{L^{\infty}(B_1)}}{(L_1\beta-2L_2-A_0)^{\theta(\bar x)}}\\
  &+ C(\lambda,\beta)\|f\|_{L^{\infty}(B_1)}^{1/(1+\theta(\bar y))}+\|f\|_{L^{\infty}(B_1)}(4L_2+2A_0)^{-\theta(\bar y)},
\end{aligned}
\end{equation}
and if $-1<\theta(\bar x)\leq 0$ and  $\theta(\bar y)>0$ we obtain 
\begin{equation}\label{contradiction2_case4}
\begin{aligned}
  3\lambda L_1\beta(1-\beta)\leq& (\lambda+(d-1)\Lambda)(4L_2+2\iota)+ C(\lambda,\beta)\|f\|_{L^{\infty}(B_1)}^{1/(1+\theta(\bar x))}\\
  &+\|f\|_{L^{\infty}(B_1)}(4L_2+2A_0)^{-\theta(\bar x)}
  +\frac{\|f\|_{L^{\infty}(B_1)}}{(L_1\beta-2L_2-A_0)^{\theta(\bar y)}}.
\end{aligned}
\end{equation}

In all cases, by choosing $L_1$ sufficiently large, \eqref{contradiction2_case1}, \eqref{contradiction2_case2},  \eqref{contradiction2_case3},  \eqref{contradiction2_case4} produce a contradiction and the lemma is proved for $\left|q\right|\,\leq\,A_0$. 
\bigskip

\noindent{\bf Step 3} Suppose $\left|q\right|>A_0$. Because $0<2-\beta<2$, \eqref{eq_masterineq} becomes
\begin{equation*}\label{eq_masterineq2}
\begin{aligned}
  4\lambda L_1\beta(1-\beta)\leq& (\lambda+(d-1)\Lambda)(4L_2+2\iota)+ f(\bar x)|q+q_{\bar x}|^{-\theta(\bar x)}\\&-f(\bar y)|q+q_{\bar y}|^{-\theta(\bar y)}.
\end{aligned}
\end{equation*}

Choose $A_0=O(L_1)$. Therefore, we can estimate 
\[
  f(\bar x)|q+q_{\bar x}|^{-\theta(x)}\leq C\frac{\|f\|_{L^\infty(B_1)}}{A_0^{(\inf_{B_1}\theta)}},
\]
and similarly,
\[
  -f(\bar y)|q+q_{\bar y}|^{-\theta(y)}\leq C\frac{\|f\|_{L^\infty(B_1)}}{A_0^{(\inf_{B_1}\theta)}},
\]
where $C$ is a positive constant. Therefore
\[
  4\lambda L_1\beta(1-\beta)\leq (\lambda+(d-1)\Lambda)(4L_2+2\iota) + C\frac{\|f\|_{L^\infty(B_1)}}{A_0^{(\inf_{B_1}\theta)}}.
\]
By choosing $L_1$ large enough, depending on $\lambda, \Lambda, d, \displaystyle\inf_{B_1}\theta$ and $L_2$ (which in turn depends only on $r$), we obtain a contradiction and the result follows.
\end{proof}

In the sequel, we develop one of the main ingredients in the realm of regularity transmission by approximation methods. Namely, a tangential path.

\section{Tangential path}\label{sct TP}

This section is solely dedicated to the proof of a key Approximation Lemma, which plays a paramount role in our forthcoming geometric argument.

\begin{lemma}[Approximation Lemma]\label{GTA-path}
Let $u\in{C}(B_1)$ be a normalized viscosity solution to
\[
	|q+Du|^{\theta(x)}F(D^2u)=f \,\,\,\, \mbox{in} \,\,\,\ B_1,
\]
where $q\in\mathbb{R}^d$ is arbitrary. Suppose A\ref{assump_operatorF}-A\ref{assump_theta1} hold. Given $\delta>0$, there exists $\varepsilon>0,$ depending only on $d, \lambda, \Lambda$ and $\sup \theta,$ such that, if
\[
	\|f[1+|q|]^{-\theta(\cdot)} \|_{L^\infty(B_1)}\leq\varepsilon,
\]
one can find $h\in {C}^{1,\bar{\alpha}}(B_{3/4})$, for some $0<\bar{\alpha}<1$, satisfying
\[
	\|u-h\|_{L^\infty(B_{1/2})}\leq\delta.
\]
Furthermore, $\|h\|_{C^{1,\bar{\alpha}}(B_{3/4})} \le C$, where $C$ depends only on $d, \lambda, \Lambda$.
\end{lemma} 
\begin{proof}
We argue by contradiction. Suppose the thesis of the proposition fails. Then there exist  $\delta_0>0$ and sequences $(F_j)_{j\in\mathbb{N}}$, $(\theta_j)_{j\in\mathbb{N}}$, $(f_j)_{j\in\mathbb{N}}$, $(q_j)_{j\in\mathbb{N}}$ and $(u_j)_{j\in\mathbb{N}}$ satisfying: 
\begin{itemize}
\item[i)] $F_j(\cdot) \ \ \mbox{is} \ \ (\lambda, \Lambda)-\mbox{elliptic}$;
\item[ii)] $-1 < \theta_j \le \sup \theta$;	
\item[iii)] $\|f_j[1+|q_j|]^{-\theta_j(\cdot)} \|_{L^\infty(B_1)} \leq\frac{1}{j}$;
\item[iv)] $|q_j+Du_j|^{\theta_j(x)}F_j(D^2u_j)=f_j$;
\end{itemize}	
however
\begin{equation}\label{edu-1}
	\displaystyle\sup_{B_{1/2}}|u_j-h|>\delta_0,
\end{equation}
for all $h\in {C}^{1,\bar{\alpha}}(B_1)$ and every $0<\bar{\alpha}<1.$

Uniform ellipticity implies that $F_j$ converges to some $(\lambda, \Lambda)$-uniform elliptic operator $F_\infty.$ In addition, the compactness proven in the previous section implies that $u_j$ converges to a function $u_\infty$ locally uniformly in $B_1$. The goal is to verify that the limiting function $u_\infty$ is a viscosity solution to the homogeneous equation
\begin{equation}\label{edu0}
	F_\infty(D^2 u_\infty) = 0, \quad \text{ in } B_{3/4}.
\end{equation}
For that, we initially rewrite the sequence of PDEs as
$$
	[1+|q_j|]^{\theta_j(x)} \cdot \left |\frac{q_j}{1+|q_j|} + \frac{1}{1+|q_j|} Du_j \right |^{\theta_j(x)}F_j(D^2u_j)=f_j.
$$
By local compactness of $\mathbb{R}^d$, up to a subsequence,
$$
	\eta_j := \frac{q_j}{1+|q_j|} \to \eta, \quad \text{and} \quad  \frac{1}{1+|q_j|} \to \mu,
$$
with $|\eta | \le 1$ and $0\le \mu \le 1$. By contradiction assumption (iii), 
$$
	g_j(x) := f_j[1+|q_j|]^{-\theta_j(x)} \to 0,
$$
in the $L^\infty$-topology.  Next, let 
$$
	p(x) = \frac{1}{2} \langle M(x-y), (x-y) \rangle  + b\cdot (x-y) + u(y)
$$ 
be a quadratic polynomial touching $u_\infty$ from below at a point $y \in B_{3/4}$. With no loss of generality, let us assume that $|y|= u(y) = 0$. Aiming at \eqref{edu0}, we need to establish 
$$
	F_\infty(M) \le 0.
$$
For $0<r_0 < 1$ fixed, define
$$
 p(x_j) - u_j(x_j) := \max\limits_{x\in B_r}( p(x) - u_j(x) ).
$$
From the PDE satisfied by $u_j$, we obtain:
\begin{equation}\label{edu1}
	 \left |\frac{q_j}{1+|q_j|} + \frac{1}{1+|q_j|} b \right |^{\theta_j(x_j)}F_j(M) \le f_j (x_j) \cdot [1+|q_j|]^{\theta_j(x_j)} = g_j(x_j). 
\end{equation}
Up to a subsequence, we can assume
$$
	 \left |\frac{q_j}{1+|q_j|} + \frac{1}{1+|q_j|} b \right |^{\theta_j(x_j)} \to \beta, \quad \text{ in } \mathbb{R}.
$$
If $\beta > 0$, then, letting $j \to \infty$ in \eqref{edu1}, we deduce, 
$$
	F_\infty(M) \le 0,
$$
as desired. The complementar case is when $b=0$ and $q_j \to 0$. From now one, we restrict to that scenario. Initially we note that assumption (iv) yields
$$
	|q_j|^{\theta_j(x_j)} F_j(M) \le f_j(x_j).
$$
If we can find subsequence $q_{j_k} \not = 0$, we reach the same conclusion, as from assumption (ii), along with the information $q_j \to 0$, we know $\|f_j |q_j|^{-\theta_j(x_j)} \|_\infty = o(1)$. 

We are left to analyze the case $b=0 = q_j$. One further reduction: if $\text{spec}(M)\subset (-\infty, 0]$, then, by ellipticity, we immediately deduce $F_\infty(M) \le 0$.  Thus, we can assume the invariant space $E:=\text{span}(e_1, e_2, \cdots, e_k)$, formed by all eigenvectors associated with positive eigenvalues is nonempty; let $\mathbb{R}^d = E \oplus G$ be an orthogonal sum. For $\kappa>0$ to be chosen small, define the test function
$$
	\phi(x) := \frac{1}{2} \langle Mx,x \rangle + \kappa \sup\limits_{e\in \mathbb{S}^{d-1}} \langle P_Ex, e \rangle,
$$
where $P_E$ is the orthogonal projection on $E$. Note, as $u_j \to u$ uniformly and $ \frac{1}{2} \langle Mx,x \rangle$ touches $u$ at $0$ from below,  $\phi$ touches $u_j$ from below at an interior point $x_j^\kappa \in B_{r}$, for all $0< \kappa \le \kappa_0$. Next, if  $x_j^\kappa \in G$, then   
$$
	\frac{1}{2} \langle Mx,x \rangle + \kappa  \langle P_Ex, e \rangle
$$
touches $u_j^\kappa$ at $x_j^\kappa$ for any choice of $e$. Thus, from the PDE satisfied by $u_j$, we obtain
\begin{equation}\label{edu2}
	|Mx_j^\kappa + \kappa e|^{\theta_j(x_j)} F_j(M) \le f_j(x_j), \quad \forall e\in \mathbb{S}^{d-1}.
\end{equation}
Since
$$
	\|M\| +\kappa \ge \sup\limits_{e\in \mathbb{S}^{d-1}}  |Mx_j^\kappa + \kappa e| \ge \kappa,
$$
taking the supremum in $e$ and subsequently letting $j\to \infty$ in \eqref{edu2}, we reach $ F_\infty(M) \le 0$. Finally, if $P_E(x_j^\kappa) \not = 0$, then
$$
	\sup\limits_{e\in \mathbb{S}^{d-1}} \langle P_E x_j^\kappa, e \rangle = \langle P_E x_j^\kappa, \frac{P_E x_j^\kappa}{|P_E x_j^\kappa|}  \rangle= |P_E x_j^\kappa|.
$$
Set
\[
	\nu  =  \frac{P_E x_j^\kappa}{|P_E x_j^\kappa|}.
\]
From the PDE satisfied by $u_j$, we reach:
\begin{equation}\label{edu3}
	|M x_j^\kappa + \kappa \nu|^{\theta_j(x_j^\kappa)} F_j \left (M + \kappa \left (\text{Id} - \nu \otimes \nu\right ) \right ) \le f_j(x_j^\kappa).
\end{equation}
Next write $x_j^\kappa$ in the basis formed by the eigenvectors of $M$, $x_j^\kappa = \sum_{i=1}^d a_i e_i$, 
so that
$$
	Mx_j^\kappa = \lambda_1 a_1e_1 + \lambda_2 a_2e_e + \cdots \lambda_k a_ke_k + \lambda_{k+1} a_{k+1} e_{k+1} +\cdots + \lambda_d a_d e_d,
$$
and, as set before, $\lambda_i > 0$ for all $i=1,2, \cdots,  k$. We  can estimate:
$$
	\begin{array}{lll}
		|M x_j^\kappa + \kappa \nu| &\ge& \langle M x_j^\kappa + \kappa \nu, \nu \rangle \\
		&=& \frac{1}{|P_E x_j^\kappa|} \langle  \sum_{i=1}^d \lambda_i a_i e_i ,\sum_{i=1}^k \lambda_i a_i e_i  \rangle + \kappa  \\
		&= & \sum_{i=1}^k \lambda_i a^2_i + \kappa\\
		&\ge& \kappa. 
	\end{array}
$$
Hence, multiplying \eqref{edu3} by $|M x_j^\kappa + \kappa \nu|^{-\theta_j(x_j^\kappa)}$  and letting $j\to \infty$, we finally reach
$$
	F_\infty(M) \le F_\infty \left (M + \kappa \left (\text{Id} - \nu \otimes \nu\right ) \right ) \le 0, 
$$
since $ \kappa \left (\text{Id} - \nu \otimes \nu\right ) \ge 0$.  This concludes the proof that $u_\infty$ is a viscosity supersolution to the equation $F_\infty(D^2h) = 0$. Arguing analogously, we obtain $u_\infty$ is also a viscosity subsolution to that equation, and thus \eqref{edu0} is proven.

It  now follows from \cite{ccbook} that $u_\infty\in {C}^{1,\bar{\alpha}}(B_{3/4})$, for some $0<\bar{\alpha}<1$ and that $\|u_\infty\|_{C^{1,\bar{\alpha}}(B_{1/2})} \le C$, where $C$ depends only on $d, \lambda, \Lambda$. Finally, taking $h = u_\infty$, we reach a contradiction on \eqref{edu-1}, for $j\gg 1$. The Lemma is finally proven. 
\end{proof}

\section{H\"older continuity of the gradient}\label{sec_c1alpha}

The approximation Lemma proven in the previous section sponsors a tangential path connecting the Krylov-Safonov regularity theory, available for the limiting profile, and the one for our problem of interest. This is the rationale behind the proof of Theorem \ref{teo_c1alpha}, which we describe in details below:

\begin{proof}[Proof of Theorem \ref{teo_c1alpha}] We start off the proof by fixing a number 
$$
	0 < \gamma <  \min \left \{\bar{\alpha}, \frac{1}{1+\|\theta^{+}\|_\infty + \|\theta^{-}\|_\infty } \right \}.
$$
Let us also choose and fix a point $y \in B_{1/2}$. We aim to establish  the existence of universal constants $0<r\ll1$, $C>1$, and a sequence of affine functions  
\[
	\ell_n(x):=a_n+b_n\cdot x,
\] 
with $(a_n)_{n\in\mathbb{N}}\subset\mathbb{R}$ and $(b_n)_{n\in\mathbb{N}}\subset\mathbb{R}^d$, verifying, for all every $n\in\mathbb{N}$, the following three estimates:

\begin{itemize}
\item[i)]  $\sup\limits_{B_{r^n}(y)}\left|u(x)\,-\,\ell_n(x)\right|\,\leq \,r^{n(1+ \gamma)}$;
\item[ii)]  $\left|a_{n-1}\,-\,a_n\right|  \leq Cr^{(n-1)(1+\gamma)}$;
\item[iii)] $ \left|b_{n-1}\,-\,b_n\right| \le Cr^{(n-1)\gamma} $
\end{itemize}
We show these by means of an induction argument.

\bigskip

\noindent \textbf{Step 1.} By a variable translation $x \mapsto y + \dfrac{1}{2}x$, we can consider $y=0$. We start by setting 
\[
	\ell_1(x):=h(0)+Dh(0)\cdot x,
\] 
where $h$ is the approximate function from Lemma \ref{GTA-path}, for  a $\delta> 0$ to be prescribed.  For a constant $C>1$, depending only on dimension and ellipticity,
$$
	\left | Dh(0) \right | + \sup\limits_{B_r} \left | h(x) - \ell_1(x) \right | \le Cr^{1+\bar{\alpha}}.
$$
Also, the triangle inequality yields
\begin{equation}\label{ind-edu1}
	\sup_{x\in B_{r}}\left|u(x)\,-\,\ell_1(x)\right|\,\leq\, \delta\,+\,Cr^{1+\bar{\alpha}}.
\end{equation}
To conclude the first step in the induction process we set 
$$
	a_1 = h(0), \quad b_1 = Dh(0),
$$
and in the sequel make two universal choices: initially we choose and fix $0<r\ll 1$ so small that the following estimates,
\begin{equation}\label{ind-edu2}
	r^\gamma \le \dfrac{1}{2}, \quad Cr^{1+\bar{\alpha}} \le \frac{1}{2} r^{1+\gamma}, \quad \text{and} \quad \left (1+4C \right )  r^{1-\gamma(1+\|\theta^{+}\|_\infty +\|\theta^{-}\|_\infty)} \le 1,
\end{equation}
are verified. Finally, we set
\begin{equation*}\label{ind-edu3}
	\delta\,:=\,\frac{r^{1+\gamma}}{2}, 
\end{equation*}
which fixes through Lemma \ref{GTA-path}, with $q=0$, the smallness condition for $\varepsilon \ge \|f\|_{L^\infty(B_1)}$. Recall from subsection \ref{subsec_small}, such a smallness assumption on  $\|f\|_{L^\infty(B_1)}$ can be assumed with no loss of generality. The fist step of induction has been verified.

\bigskip

\noindent\textbf{Step 2.} Suppose the induction hypotheses have been established for $n=1,2, \cdots, k$, for some $k\in\mathbb{N}$. We must show the case $n=k+1$ also holds true. 
For that, we introduce the auxiliary function:
\[
  v_k(x):=\frac{u(r^kx)\,-\,\ell_k(r^kx)}{r^{k(1+\gamma)}}.
\]
Notice that $v_k$ solves
\[
  |r^{-k\gamma}b_k+Dv_k|^{\theta_k(x)}F_k(D^2v_k)=f_k(x)
\]
where 
\[
	F_k(M)=r^{k(1-\gamma)}F(r^{(\gamma-1)k}M),
\]
is a $(\lambda, \Lambda)$-elliptic operator, $\theta_k(x)=\theta(r^kx)$ and
\[
f_k(x)=r^{k(1-\gamma)-k\gamma\theta_k(x)}f(r^kx).
\]
From our choice for $\gamma$ in \eqref{ind-edu1}, we have $f_k\in L^\infty(B_1)$. Initially we estimate
$$
	\begin{array}{lll}
		\left | b_k \right | &=& \left | b_1 + \sum\limits_{i=1}^k \left ( b_i - b_{i-1} \right ) \right |  \\
		&\le &  \left | b_1 \right | + \sum\limits_{i=1}^k \left | b_i - b_{i-1} \right | \\
		&\le &  C + \sum\limits_{i=1}^k  Cr^{(k-1)\gamma} \\
		& \le &  \frac{2C}{1-r^{\gamma}}  \\
		&\le & 4C,
	\end{array}
$$
in view of first estimate required in \eqref{ind-edu2}. Next, we note 
$$
	-\theta_k(x) \le \|\theta^{-}\|_\infty < 1.	
$$
Thus we can further estimate
$$
	\begin{array}{lll}
		\left \|f_k(x) \left [1+ |r^{-k \gamma}b_k| \right ]^{-\theta_k(x)} \right \|_\infty &\le& \left \|f_k(x) \left [1+ |r^{-k \gamma}b_k|^{\|\theta^{-}\|_\infty}  \right ]\right \|_\infty \\
		&\le & \varepsilon \left [  r^{(1-\gamma)-\gamma\|\theta^+\|_\infty} + 4C r^{(1-\gamma)-\gamma\|\theta^+\|_\infty - \gamma\|\theta^-\|_\infty} \right ] \\
		&\le & \varepsilon ,
	\end{array}
$$
in accordance to the third estimate enforced in \eqref{ind-edu2}. We have proved $v_k$ is under the assumptions of Lemma \ref{GTA-path}, which assures the existence of a function $\bar{h}\in {C}^{1,\bar{\alpha}}(B_1)$ such that
\[
  \displaystyle\sup_{x\in B_{r}}|v_k(x)-\bar{h}(x)|\leq \delta.
\]
As in Step 1, we estimate
\[
  \displaystyle\sup_{x\in B_{r}}|v_k(x)-\bar{\ell}(x)|\leq r^{1+\gamma},
\]
with 
\[
  \bar{\ell}(x)\,=\,\bar{a}\,+\,\bar{b}\cdot x.
\]
Setting
\[
  \ell_{k+1}(x):= \ell_k(x)\,+\,r^{k(1+\gamma)}\bar{\ell}(r^{-k}x).
\]
yields
\[
	\sup_{x\in B_{r^{k+1}}}\left|u(x)\,-\,\ell_{k+1}(x)\right|\,\leq \,r^{(k+1)(1+\gamma)}.
\]
Also,
\[
	\left|a_{k+1}\,-\,a_k\right|\,+\,r^k\left|b_{k+1}\,-\,b_k\right|\,\leq\, Cr^{k(1+\gamma)},
\]
and thus the $(k+1)$-th step in the induction is complete.

\bigskip

\noindent \textbf{Step 3.} Both  sequences $(a_n)_{n\in\mathbb{N}}$ and $(b_n)_{n\in\mathbb{N}}$ furnished through the induction process are Cauchy sequences and therefore converge. Let us label  
$$
	a_n \to a^{\star} \quad \text{ and } \quad b_n \to b^{\star}.
$$ 
 Evaluating estimate 
 $$
 	\sup\limits_{x\in B_{r^n}}\left|u(x)\,-\,\ell_n(x)\right|\,\leq \,r^{n(1+ \gamma)}
$$ 
on $x=0$ and letting $n\to \infty$ yields $a^{\star}=u(0).$ On the other hand, estimate $ \left |b_{n-1}\,-\,b_n\right| \le Cr^{(n-1)\gamma} $ gives
\[
  |b_n-b^{\star}|\leq Cr^{n\gamma}.
\]
To conclude the proof,  let $0<\rho<r$ and take the first integer $n\in\mathbb{N}$ for which $r^{n+1}\,<\rho\,\leq r^n.$ We can estimate
\begin{equation*}
\begin{aligned}
  \displaystyle\sup_{x\in B_{\rho}}|u(x)-(u(0)+b^{\star}\cdot x)|&\leq\displaystyle\sup_{x\in B_{r^n}}|u(x)-\ell_n(x)|+\displaystyle\sup_{x\in B_{r^n}}|a_n-u(0)|\\&\quad+\displaystyle\sup_{x\in B_{r^n}}|b_n-b^{\star}||x|\\
  &\leq\frac{C}{r}r^{(n+1)(1+\gamma)}\\
  &\leq C\rho^{1+\gamma}.
\end{aligned}
\end{equation*}
Now, notice that, since $u$ is continuous in $B_1,$ for $x\in\overline{B_{\rho}}$ such that $|x|=\rho$ we obtain  
\begin{equation}\label{u diffe}
|u(x)-(u(0)+b^{\star}\cdot x)|\leq C |x|^{1+\gamma}.
\end{equation}
But, since $\rho$ is arbitrary, $0<\rho<r,$ estimate \eqref{u diffe} is true for every $x\in B_r.$ Hence,
\[
u(x)=u(0)+b^{\star}\cdot x+o(|x|),
\]
which implies that $u$ is differentiable at $0$ and that $b^{\star}=Du(0).$ If we now go back to the original point $y\in B_{1/2}$, we conclude the estimate
\begin{equation}\label{pointwise diff}
	\displaystyle\sup_{x\in B_{\rho}(y) }|u(x)-  \left (u(y)+Du(y)\cdot (x-y) \right )|\leq C\rho^{1+\gamma}
\end{equation}
holds for all $0< \rho < 1/2$. 

\bigskip

\noindent \textbf{Step 4.} We finally obtain H\"older continuity of the gradient from \eqref{pointwise diff}. Given two points $y_1 \not = y_2 \in B_{1/2}$, estimate \eqref{pointwise diff} can be written as
\begin{equation}\label{local est1}
	 u(y_2) - u(y_1) - Du(y_1) \cdot (y_1 - y_2)   = \text{O} \left (  |y_1 - y_2|^{1+\gamma} \right ).
\end{equation}
Similarly, 
\begin{equation}\label{local est2}
	  u(y_1) - u(y_2) - Du(y_2) \cdot (y_2 - y_1)  =\text{O} \left (  |y_1 - y_2|^{1+\gamma} \right ).
\end{equation}
Adding these two estimates together and interchanging the roles of $y_1$ and $y_2$ gives
\begin{equation}\label{local est3}
	\left | \partial_\nu u(y_1)  - \partial_\nu u(y_2)  \right |= \text{O}\left (  |y_1 - y_2|^{ \gamma} \right ),
\end{equation}
where $\nu := \frac{y_1 - y_2}{|y_1-y_2|}$. If we subtract \eqref{local est1} from \eqref{local est2}, taking into account \eqref{local est3}, we obtain
\begin{equation}\label{local est35}
	 \left | u(y_1) - u(y_2) \right | +  \left |\partial_\nu u(y_1)  + \partial_\nu u(y_2) \right | \cdot |y_1-y_2| = \text{O} \left (  |y_1 - y_2|^{1+\gamma} \right ).
\end{equation}

Next let $\nu_2, \nu_3, \cdots, \nu_d$ be unitary vectors such that $\{\nu, \nu_2, \cdots, \nu_d\}$ is an orthonormal basis of $\mathbb{R}^d$. Applying \eqref{pointwise diff} at $y_1$ for 
$$
	x_i :=  \frac{y_1 + y_2}{2} +  \frac{\sqrt{3}\left |y_1-y_2 \right |}{2}\nu_i 
$$
gives
\begin{equation}\label{local est4}
	 u(x_i) = u(y_1) - Du(y_1) \cdot  (x_i - y_1)  +  \text{O} \left (  |y_1 - y_2|^{1+\gamma} \right ),
\end{equation}
and likely, applying at $y_2$ for the same $x_i$ yields
\begin{equation}\label{local est5}
	 u(x_i) = u(y_2) - Du(y_2) \cdot (x_i - y_2)  +  \text{O} \left (  |y_1 - y_2|^{1+\gamma} \right ).
\end{equation}
Thus, subtracting \eqref{local est5} from \eqref{local est4}, we end up with
\begin{equation}\label{local est6}
	\left (u(y_1) - u(y_2) \right ) + Du(y_2) \cdot (x_i - y_2) - Du(y_1) \cdot (x_i - y_1) =  \text{O} \left (  |y_1 - y_2|^{1+\gamma} \right ).
\end{equation}
Finally, we observe that,
\begin{equation}\label{local est7}
	Du(y_1) \cdot (x_i - y_1) = - \partial_\nu u(y_1) \frac{ |y_1 - y_2|}{2} + \frac{\sqrt{3}|y_1 - y_2| }{2} \partial_{\nu_i} u(y_1) 
\end{equation}
and 
\begin{equation}\label{local est9}	
	Du(y_2) \cdot (x_i - y_2) =  \partial_\nu u(y_2) \frac{ |y_1 - y_2|}{2} + \frac{\sqrt{3}|y_1 - y_2| }{2} \partial_{\nu_i} u(y_2).
\end{equation}
Together with \eqref{local est35} and \eqref{local est6}, the expressions in \eqref{local est7} and \eqref{local est9} produce
\begin{equation*}\label{local est8}	
	\partial_{\nu_i} u(y_2) - \partial_{\nu_i} u(y_1) =  \text{O} \left (  |y_1 - y_2|^{\gamma} \right ),
\end{equation*}
for all $i=2,3,4, \cdots, d$. Then, the proof of the theorem is complete.
\end{proof}

\section{Sharp pointwise estimates} \label{Sharp sct}

In this final section we discuss sharp, improved geometric estimates in the spirit of \cite{T1, T3}. We start off with  a consequence of our main Theorem \ref{teo_c1alpha} in the pure singular case.

\begin{corollary} \label{cor1}
Let $u\in {C}(B_1)$ be a viscosity solution to \eqref{mainequation}. Suppose A\ref{assump_operatorF} and A\ref{assump_theta1} are in force. Assume further $\theta(x) \le 0$.
Then $u\in {C}^{1,\beta}_{loc}(B_1)$, for all  $\beta$ verifying 
$$
	\beta<  \bar{\alpha},
$$
where $\bar{\alpha}$ is the sharp H\"older continuity exponent associated with $F$-harmonic functions. In particular,  $u\in {C}^{1,\text{LogLip}}_{loc}(B_1)$, provided $F$ is convex.
\end{corollary}
\begin{proof}
	It follows from Theorem \ref{teo_c1alpha} that $u\in {C}^{1,\gamma}_{loc}(B_1)$ for $0<\gamma \ll 1$, as described therein. In particular $|Du| \in L^\infty_\text{loc}(B_1)$. Arguing as in \cite{BD}, one can verify that $u$ solves
	$$
		F(D^2 u) = f(x) \cdot |Du|^{-\theta(x)} =: g(x) \in L^\infty_\text{loc}(B_1),
	$$
	in the viscosity sense. Corollary \ref{cor1} now follows from \cite{C1}. When $F$ is convex,  ${C}^{1,\text{LogLip}}_{\text{loc}}(B_1)$ follows from  \cite{T2}. 
\end{proof}

Next we would like to investigate geometric growth estimates on $u$ at a given point $y \in B_{1/2}$.  Well, if, say $|Du(y)| \ge \tau >0$, then by Theorem  \ref{teo_c1alpha}, there exists a radius $r_\tau>0$, depending only on $\tau$ and universal parameters, such that $|Du| \ge \frac{1}{2}\tau_0$ in $B_{r_\tau}(y)$. Thus the solution $u$ has an asymptotic growth behavior of an $F$-harmonic function at $y$; that is:
$$
	\sup\limits_{x\in B_r(y)} \left | u(x) - \left ( u(y) + Du(y) \cdot (x-y) \right )\right | \le C r^{1+\gamma},
$$
for all $\gamma <  \bar{\alpha}$ and all $0<r< r_\tau$. The really interesting case is when $y$ is a critical point of $u$.  Hereafter estimates shall depend on the modulus of continuity of the variable exponent $\theta$. More precisely, we will impose the following condition on $\theta$:

\begin{assumption}[Continuity of $\theta$]\label{assump_theta2} 
We assume 
$$
	\left | \theta(x) - \theta(y)  \right | \le L \omega \left ( \left |x-y \right | \right ), 
$$
for a modulus of $\omega \colon \mathbb{R}_{+} \to  \mathbb{R}_{+}$ satisfying $\omega(1) = 1$.
\end{assumption}

Under condition A\ref{assump_theta2}, it follows from Corollary \ref{cor1} that if $\theta(y) < 0$ and, say, $F$ is convex, then at $y$ $u$ grows in a  ${C}^{1,\text{LogLip}}$ fashion. In what follows, we pursue fine geometric estimates for solutions at degenerate elliptic, critical points. To simplify the presentation, we shall assume hereafter convexity of $F$, under which Evans--Krylov theorem assures $F$-harmonic functions are of class $C^{2,\alpha}$ for some $\alpha>0$, see \cite[Chapter 6]{ccbook}.

\begin{theorem} \label{sharp1}
Let $u\in {C}(B_1)$ be a viscosity solution to \eqref{mainequation}. Suppose A\ref{assump_operatorF}, A\ref{assump_theta1}, and A\ref{assump_theta2} are in force and $F$ is convex. Assume further that $Du(y) = 0$, for some $y \in B_{1/2}$ and that $\theta(y) > 0$. Then, for any $p > \theta(y)$, there holds
$$
	\sup\limits_{x\in B_r(y)} \left | u(x) -  u(y) \right | \le C r^{1+ \frac{1}{1+p}},
$$
for a constant $C>1$ that depends only on $\left ( p-\theta(y) \right )$, the dimension $d$, $\|u\|_\infty$, ellipticity constants, and the modulus of continuity of $\theta$, meaning $L$ and $\omega$.
\end{theorem}
\begin{proof} We start off the proof by assuming, with no loss of generality, $\|u\|_\infty = 1$, $y = u(y) = 0$. We shall divide the proof into two steps as to better convey the reasoning behind the arguments.

\bigskip 

\noindent \textbf{Step 1.}
We claim that given $\delta > 0$ there exists $\epsilon> 0$, depending only on $\delta> 0$, the dimension $d$, ellipticity constants and the modulus of continuity of $\theta$, such that if $\|f\|_{L^\infty(B_1)} \le \epsilon$, then
\begin{equation*}\label{sharp eq1}
	\sup\limits_{B_{1/2}} \left | v(x) - h(x) \right | < \delta
\end{equation*}
for some $F$-harmonic function $h$ verifying $Dh(0) = 0$. 

 As in Lemma \ref{GTA-path}, we shall establish this by \textit{reductio ad absurdum}. If this is not the case, then there exist  $\delta_0>0$ and sequences $(F_j)_{j\in\mathbb{N}}$, $(\theta_j)_{j\in\mathbb{N}}$, $(f_j)_{j\in\mathbb{N}}$, and $(u_j)_{j\in\mathbb{N}}$ satisfying: 
\begin{itemize}
\item[i)] $F_j(\cdot) \ \ \mbox{is} \ \ (\lambda, \Lambda)-\mbox{elliptic}$;
\item[ii)] $-1 < \theta_j \le \sup \theta$ and $\theta_j$ has the same modulus of continuity of $\theta$;	
\item[iii)] $\|f_j\|_{L^\infty(B_1)} \leq\frac{1}{j}$;
\item[iv)] $\|u_j\|_{L^\infty(B_1)} =1$; $|D_j(0)| = 0$;
\item[v)] $|Du_j|^{\theta_j(x)}F_j(D^2u_j)=f_j$;
\end{itemize}	
however
\begin{equation}\label{sharp eq2}
	\displaystyle\sup_{B_{1/2}}|u_j-h|>\delta_0,
\end{equation}
for all $F$-harmonic function $h$ for which $0$ is a critical point and $h(0) = 0$. It follows from Theorem \ref{teo_c1alpha} that, up to a subsequence
$$
	(u_j, Du_j) \to (v_\infty, Dv_\infty),
$$
locally uniformly in $B_{1/2}$. In particular $v_\infty(0) = \left | Dv_\infty(0) \right |= 0$. Also, from the Arzel\`a-Ascoli Theorem, up to a subsequence, $\theta_j \to \theta_\infty$ and $F_j \to F_\infty$ locally uniformly. Taking the limit as $j \to \infty$ in (v), we conclude
$$
	|D v_\infty|^{\theta_\infty(x)} F_\infty(D^2 v_\infty) = 0,
$$
and arguing as in \cite{Imb-Silv1}, we deduce $F_\infty(D^2 v_\infty) = 0$. We arrive to a contradiction by taking $h= v_\infty$ in \eqref{sharp eq2}.

\bigskip

\noindent \textbf{Step 2.}
Next, for $p > \theta(0)$ fixed and for $\lambda>0$ to be determined, we estimate
$$
	\begin{array}{lll}
		\sup\limits_{B_\lambda} |u(x)| &\le& \sup\limits_{B_\lambda} |u(x) - h(x)| + \sup\limits_{B_\lambda} |h(x)|   \\
		& \le & \delta + C \lambda^{\frac{3}{2}+ \frac{1}{2\left ( 1+p \right )}} \\
		& \le & \delta + \frac{1}{2} \lambda^{1+ \frac{1}{1+p}} 
	\end{array}
$$
provided we choose $\lambda$ so small that 
$$
	 C \lambda^{\frac{3}{2}+ \frac{1}{2\left ( 1+p \right )}} \le \frac{1}{2} \lambda^{1+ \frac{1}{1+p}}.
$$ 
From A\ref{assump_theta2}, we can take $\lambda$ even smaller, as to assume
\begin{equation}\label{sharp eq3}
	 \theta(\lambda x) < p,
\end{equation}
for all $x\in B_1$. In the sequel, we select
$$
	 \delta \le \frac{1}{2} \lambda^{1+ \frac{1}{1+p}}
$$
which determines the smallness condition on $\|f\|_\infty$ by means of the initial claim.  We have obtained
\begin{equation}\label{sharp eq4}
	\sup\limits_{B_\lambda} |u(x)| \le \lambda^{1+ \frac{1}{1+p}}.
\end{equation}
The idea now is to rescale this inequality to the unit picture and apply it recursively. That is, the rescaled function
$$
	v(x) := \frac{u(\lambda x)}{\lambda^{1+ \frac{1}{1+p}}},
$$
is clearly normalized, by \eqref{sharp eq4}, and also verifies
$$
	|D v|^{\theta(\lambda x)} F_\lambda(D^2 v) = \lambda^{1 - \frac{1+\theta(\lambda x)}{1+p}} f(\lambda x)  =: g(x),
$$
in the viscosity sense. From \eqref{sharp eq3}, $\|g\|_\infty \le \epsilon$, and thus $v$ is entitled to the same conclusions as $u$. In particular, \eqref{sharp eq4} holds for $v$, which leads to
\begin{equation*}\label{sharp eq5}
	\sup\limits_{B_{\lambda^2}} |u(x)| \le \lambda^{2\left (1+ \frac{1}{1+p} \right )}.
\end{equation*}
Continuing the process inductively yields the conclusion of Theorem \ref{sharp1}.
\end{proof}

Theorem \ref{sharp1} gives an asymptotically sharp geometric growth estimate of $u$ at a degenerate critical point. Such a result is indeed optimal, unless we require further control on the modulus of continuity of $\theta$. In our final theorem we show that $u$ grows precisely as a ${C}^{\frac{2+\theta(y)}{1+\theta(y)}}$ function, provided $\theta$ is Dini continuous. More precisely, we shall require:

\begin{assumption}[Critical continuity of $\theta$]\label{assump_theta3} 
We assume 
$$
	\left | \theta(x) - \theta(y)  \right | \le L_1 \omega \left ( \left |x-y \right | \right ), 
$$
for a modulus of $\omega \colon \mathbb{R}_{+} \to  \mathbb{R}_{+}$ satisfying $\omega(1) = 1$ and that
\begin{equation}\label{CMC}
	\limsup\limits_{t\to 0} \omega(t) \ln \left( \frac{1}{t} \right ) \le L_2.
\end{equation}
\end{assumption}

\smallskip

While \eqref{CMC} is indeed very mild, we call assumption A\ref{assump_theta3} {\em critical} because this condition seems to play a central role in the regularity theory for variable exponent PDEs. For apparently very different reasons, assumption  A\ref{assump_theta3} is captious in the variational theory of functionals with $p(x)$-growth. For instance, it is shown in \cite{V97} that $p(x)$-growth functionals exhibit the so called Lavrentiev
phenomenon if and only if A\ref{assump_theta3} is violated. In  \cite{ABF} it is proved that the
singular part of the measure representation of relaxed integrals with $p(x)$-growth 
fades out if and only if A\ref{assump_theta3} holds true.

\smallskip

We are ready to establish sharp geometric estimate at critical points of viscosity solutions to  \eqref{mainequation}; compare it with the $C^{p'}$-regularity conjecture for functions whose $p$-laplacian are bounded, \cite{ATU1, ATU2}.

\begin{theorem} \label{sharp2}
Let $u\in {C}(B_1)$ be a viscosity solution to \eqref{mainequation}. Suppose A\ref{assump_operatorF}, A\ref{assump_theta1}, and A\ref{assump_theta3} are in force and that $F$ is convex. Assume further that $Du(y) = 0$, for some $y \in B_{1/2}$ and that $\theta(y) > 0$. Then,  
$$
	\sup\limits_{x\in B_r(y)} \left | u(x) -  u(y) \right | \le C r^{1+ \frac{1}{1+\theta(y)}},
$$
where $C$ depends only on the dimension $d$, $\|u\|_\infty$, ellipticity constants, and the modulus of continuity of $\theta$, meaning $L_1$, $L_2$ and $\omega$.
\end{theorem}
\begin{proof}
We revisit the proof of Theorem \ref{sharp1}, owning assumption A\ref{assump_theta3}. Initially, for future use,  we estimate 
\begin{equation}\label{sharp2 eq11}
	\lambda^{1 - \frac{1+\theta(\lambda x)}{1+ \theta(0)}}  \le \left ( \left (\lambda^{-1} \right )^{L_1\omega(\lambda)} \right ) ^{\frac{1}{1+\theta(0)}},
\end{equation}
for all $x\in B_1$. Assumption A\ref{assump_theta3} yields the existence of $0< \lambda_0 \ll 1$ such that
$$
	\begin{array}{lll}
	 	\left (\lambda^{-1} \right )^{\omega(\lambda)} &\le & \left (\lambda^{-1} \right )^{\sqrt{2} L_2 \ln (\lambda^{-1})} \\
		&\le & e^{ \sqrt{2}L_2},
	\end{array}
$$
and thus
\begin{equation}\label{sharp2 eq2}
	\lambda^{1 - \frac{1+\theta(\lambda x)}{1+ \theta(0)}} \le \Xi,
\end{equation}
for all $0 < \lambda \le \lambda_0$, where $\Xi$ depends on $L_1$ and $L_2$.  Arguing as in Step 2 of the proof of Theorem \ref{sharp1}, we can establish
\begin{equation}\label{sharp2 eq4}
	\sup\limits_{B_\lambda} |u(x)| \le \lambda^{1+ \frac{1}{1+\theta(0)}},
\end{equation}
provided $\|f\|_\infty \le \epsilon$. Revisit Subsection \ref{subsec_small} and enforce 
\begin{equation}\label{sharp2 eq5}
	\|f\|_\infty \le \frac{\epsilon}{\Xi+1},
\end{equation}
 so clearly \eqref{sharp2 eq4} still holds. Now, when  we define the rescaled function 
$$
	v(x) := \frac{u(\lambda x)}{\lambda^{1+ \frac{1}{1+\theta(0)}}},
$$
we see that
$$
	|D v|^{\theta(\lambda x)} F_\lambda(D^2 v) = \lambda^{1 - \frac{1+\theta(\lambda x)}{1+\theta(0)}} f(\lambda x)  =: g(x).
$$
Combining \eqref{sharp2 eq2} and \eqref{sharp2 eq5}, we conclude $\|g\|_\infty \le \epsilon$, so $v$ is entitled for \eqref{sharp2 eq4}, which, as before, yields
\begin{equation*}\label{sharp2 eq42}
	\sup\limits_{B_{\lambda^2}} |u(x)| \le \lambda^{2 \left (1+ \frac{1}{1+\theta(0)} \right )},
\end{equation*}
Following the argument, we define
$$
	v_2(x) := \frac{u(\lambda^2 x)}{\lambda^{2 \left ( 1+ \frac{1}{1+\theta(0)}\right )}},
$$
which is normalized and solves
$$
	|D v_2|^{\theta(\lambda^2 x)} F_{\lambda^2} (D^2 v_2) = \left (\lambda^2 \right )^{1 - \frac{1+\theta(\lambda^2 x)}{1+\theta(0)}} f(\lambda^2 x)  =: g_2(x).
$$
Since $\lambda^2 < \lambda \le \lambda_0$, once more \eqref{sharp2 eq2} combined with \eqref{sharp2 eq5} yields $\|g_2\|_\infty \le \epsilon$. That means $v_2$ is also entitled to \eqref{sharp2 eq4}. Hence,
\begin{equation*}\label{sharp2 eq43}
	\sup\limits_{B_{\lambda^3}} |u(x)| \le \lambda^{3 \left (1+ \frac{1}{1+\theta(0)} \right )}.
\end{equation*}
Continuing the process recursively, one concludes the proof of Theorem \ref{sharp2}.
\end{proof}

\begin{remark} In both Theorems \ref{sharp1} and \ref{sharp2} all we need is $C^{1,1^{-}}$ a priori estimates for $F$-harmonic functions, rather than convexity on $F$. This is obtained, for instance, by requesting only the \textit{recession} operators 
$$
	F^{\star}(M) = \lim\limits_{\delta \searrow 0} \delta \cdot  F(\delta^{-1}M), 
$$ 
to be convex, see \cite{PT, ST} for details. If $F$ is an arbitrary fully nonlinear elliptic operator, both Theorems still yield improved estimates, which are naturally limited by the $C^{1,\bar{\alpha}}$ regularity of $F$-harmonic functions. For instance, if $F$ is an arbitrary elliptic operator, then Theorem \ref{sharp2} gives geometric growth estimate of order $\text{O} \left (\frac{2+\gamma}{1+\gamma} \right )$ for all $\gamma \in (0, \bar{\alpha}) \cap (0, \theta(y)]$. 

\end{remark}


\def\polhk#1{\setbox0=\hbox{#1}{\ooalign{\hidewidth\lower1.5ex\hbox{`}\hidewidth\crcr\unhbox0}}}

\Addresses

\end{document}